\documentclass[final,onefignum,onetabnum]{siamart220329}

\usepackage{braket,amsfonts}

\usepackage{amsmath,amssymb}
\usepackage{booktabs}
\usepackage{siunitx}

\usepackage{array}

\usepackage[caption=true]{subfig}
\usepackage{caption}
\usepackage{pgfplots}

\pgfplotsset{compat=1.18}

\usepackage{hyperref}

\newsiamthm{claim}{Claim}
\newsiamremark{remark}{Remark}
\newsiamremark{hypothesis}{Hypothesis}
\crefname{hypothesis}{Hypothesis}{Hypotheses}

\usepackage[noend]{algpseudocode}

\usepackage{graphicx,epstopdf}

\Crefname{ALC@unique}{Line}{Lines}

\usepackage{amsopn}

\usepackage{xspace}
\usepackage{bold-extra}
\usepackage[most]{tcolorbox}

\colorlet{texcscolor}{blue!50!black}
\colorlet{texemcolor}{red!70!black}
\colorlet{texpreamble}{red!70!black}
\colorlet{codebackground}{black!25!white!25}

\lstdefinestyle{siamlatex}{%
  style=tcblatex,
  texcsstyle=*\color{texcscolor},
  texcsstyle=[2]\color{texemcolor},
  keywordstyle=[2]\color{texemcolor},
  moretexcs={cref,Cref,maketitle,mathcal,text,headers,email,url},
}

\tcbset{%
  colframe=black!75!white!75,
  coltitle=white,
  colback=codebackground, 
  colbacklower=white, 
  fonttitle=\bfseries,
  arc=0pt,outer arc=0pt,
  top=1pt,bottom=1pt,left=1mm,right=1mm,middle=1mm,boxsep=1mm,
  leftrule=0.3mm,rightrule=0.3mm,toprule=0.3mm,bottomrule=0.3mm,
  listing options={style=siamlatex}
}

\newtcblisting[use counter=example]{example}[2][]{%
  title={Example~\thetcbcounter: #2},#1}

\newtcbinputlisting[use counter=example]{\examplefile}[3][]{%
  title={Example~\thetcbcounter: #2},listing file={#3},#1}

\DeclareTotalTCBox{\code}{ v O{} }
{ 
  fontupper=\ttfamily\color{black},
  nobeforeafter,
  tcbox raise base,
  colback=codebackground,colframe=white,
  top=0pt,bottom=0pt,left=0mm,right=0mm,
  leftrule=0pt,rightrule=0pt,toprule=0mm,bottomrule=0mm,
  boxsep=0.5mm,
  #2}{#1}

\patchcmd\newpage{\vfil}{}{}{}
\begin{tcbverbatimwrite}{tmp_\jobname_header.tex}
\title{THREE-PRECISION ITERATIVE REFINEMENT WITH PARAMETER REGULARIZATION AND PREDICTION FOR SOLVING LARGE SPARSE LINEAR SYSTEMS\thanks{Submitted to the editors DATE.
\funding{National Key Research and Development Program of China (Grant No. 2023YFB3001604).}}}

\author{ Jifeng Ge\footnotemark[2] \and Juan Zhang\thanks{Key Laboratory of Intelligent Computing and Information Processing of Ministry of Education,
Hunan KeyLaboratory for Computation 
and Simulation in Science and Engineering, School of Mathematics and Computational Science, Xiangtan University, 
Xiangtan, Hunan, China, 411105.
(Corresponding authors. \email{zhangjuan@xtu.edu.cn}(JZ)).} }

\headers{Three mixed precision refinement algorithm}{Jifeng Ge, and Juan Zhang}
\end{tcbverbatimwrite}
\input{tmp_\jobname_header.tex}

\ifpdf
\hypersetup{ pdftitle={Guide to Using  SIAM'S \LaTeX\ Style} }
\fi

\begin{document}
\maketitle

\begin{tcbverbatimwrite}{tmp_\jobname_abstract.tex}
\begin{abstract}
    This study presents a novel mixed-precision iterative refinement algorithm, 
    GADI-IR, within the general alternating-direction implicit (GADI) framework, 
    designed for efficiently solving large-scale sparse linear systems. 
    By employing low-precision arithmetic, particularly half-precision (FP16), 
    for computationally intensive inner iterations, the method achieves 
    substantial acceleration while maintaining high numerical accuracy. 
    Key challenges such as overflow in half-precision and convergence issues 
    for low precision are addressed through careful backward error analysis 
    and the application of a regularization parameter $\alpha$. Furthermore, 
    the integration of low-precision arithmetic into the parameter 
    prediction process, using Gaussian Process Regression (GPR), 
    significantly reduces computational time without degrading performance. 
    The method is particularly effective for large-scale linear systems 
    arising from discretized partial differential equations and 
    other high-dimensional problems, where both accuracy and efficiency 
    are critical. Numerical experiments demonstrate that the use of 
    mixed-precision strategies not only accelerates computation 
    but also ensures robust convergence, making the approach advantageous 
    for various applications. The results highlight the potential 
    of leveraging lower-precision arithmetic to achieve superior 
    computational efficiency in high-performance computing.
\end{abstract}


\end{tcbverbatimwrite}
\input{tmp_\jobname_abstract.tex}

\section{Introduction}
\label{sec:intro}

\subsection{Background}
Mixed precision techniques have been a focus of research 
for many years. With advancements in hardware, such as 
the introduction of tensor cores in modern GPUs, 
half-precision arithmetic (FP16) has become 
significantly faster than single or double precision
\cite{nvidia_hopper_2024}\cite{nvidia_tensor_core_gpu_2024}, 
driving its growing importance in high-performance computing 
(HPC) and deep learning. By strategically utilizing FP16 
alongside the capabilities of modern tensor cores, 
mixed precision computing marks 
a major breakthrough in HPC. It offers 
notable improvements in computational speed, 
memory utilization, and energy efficiency, 
all while preserving the accuracy required 
for a wide array of scientific and 
engineering applications.

Solving large sparse linear systems 
is a cornerstone of numerical computing, 
with broad applications in scientific simulations, 
engineering problems, and data science. These systems frequently 
arise from the discretization of partial differential 
equations or the modeling of complex processes, where 
efficient and scalable solution methods are essential. 
In this context, the paper \cite{doi:10.1137/21M1450197} introduces an innovative and 
flexible framework called the general alternating-direction 
implicit (GADI) method. This framework tackles the challenges 
associated with large-scale sparse linear systems by unifying 
existing methods under a general structure and integrating 
advanced strategies, such as Gaussian process regression (GPR)\cite{NIPS1995_7cce53cf}, 
to predict optimal parameters. These innovations significantly 
enhance the computational efficiency, scalability, and robustness 
of solving such systems.

The motivation for incorporating mixed precision into the GADI framework stems from the growing need to balance computational efficiency and resource utilization when solving large-scale sparse linear systems. Mixed precision methods, which combine high precision (e.g., double precision) and low precision (e.g., single or half precision) arithmetic, have gained traction due to advancements in modern hardware, such as GPUs and specialized accelerators, which are optimized for lower-precision computations.

The use of low precision (like 16-bit floating-point) in the GADI framework offers a powerful combination of speed, memory efficiency, and energy savings, making it particularly well-suited for solving large sparse linear systems. FP16 computations are significantly faster than FP32 or FP64 on modern hardware like GPUs with Tensor Cores, enabling rapid execution of matrix operations while dramatically reducing memory usage, which allows larger problems to fit within the same hardware constraints. This reduced precision also lowers energy consumption, making it a more sustainable option for large-scale computations. While FP16 has limitations in range and precision, its integration into a mixed precision approach within the GADI framework ensures critical calculations retain higher precision to guarantee robustness and convergence. Additionally, FP16 accelerates the Gaussian Process Regression (GPR) used for parameter prediction, enabling efficient optimization of the framework's performance. Together, these advantages position FP16 as a key enabler for scalable and efficient numerical computing in modern applications.

Iterative refinement is a numerical technique used 
to improve the accuracy of a computed solution 
to a linear system of equations, particularly when 
the initial solution is obtained using approximate methods. 
It is widely used in scenarios where high precision is 
required, such as solving large-scale or ill-conditioned linear systems.

The process starts with an approximate solution, 
often computed in low precision for efficiency, 
followed by a series of iterative corrections. 
Each iteration involves:
\begin{itemize}
    \item Residual Computation: Calculate the residual 
    $r=Ax-b$, in precision $u_f$, where $A$ is the coefficient matrix,
    x is the current solution.
    \item Solving a linear system to obtain a correction 
    vector that reduces the residual in precision $u_r$.
    \item Updating the solution with the correction vector 
    in precision $u$.  
\end{itemize}

This process repeats until the residual or the correction 
falls below a specified tolerance, indicating convergence 
to the desired accuracy. The error analyses
were given for fixed point arithmetic by \cite{10.1145/321386.321394} 
and \cite{Wilkinson2023} for floating 
point arithmetic.

Half precision(16-bit) floating point arithmetic, defined
as a storage format in
the 2008 revision of the IEEE standard \cite{8766229}, 
is now starting to become available in
hardware, for example, in the NVIDIA A100 and H100 GPUs
\cite{nvidia_hopper_2024}\cite{nvidia_tensor_core_gpu_2024}, 
on which it runs up to 50x as fast as double precision arithmetic
with a proportional saving in energy consumption. Moreover, 
IEEE introduced the FP8 (Floating Point 8-bit) format as part 
of the IEEE 754-2018 standard revision which is designed with a 
reduced bit-width compared to standard IEEE floating-point 
formats like FP32 (32-bit floating point) and FP64 
(64-bit floating point). So in the 2010s iterative refinement 
attracted renewed interest. The following table summarizes key parameters 
for IEEE arithmetic precisions.

\begin{table}[htbp]
    \centering
    \begin{tabular}{|l|l|l|l|}
    \hline
    Type & Size & Range & Unit roundoff $u$ \\
    \hline
    half & 16bits & $10^{\pm5}$ & $2^{-11} \approx 4.9 \times 10^{-4}$ \\
    single & 32bits & $10^{\pm38}$ & $2^{-24} \approx 6.0 \times 10^{-8}$ \\
    double & 64bits & $10^{\pm308}$ & $2^{-53} \approx 1.1 \times 10^{-16}$ \\
    quadruple & 128bits & $10^{\pm4932}$ & $2^{-113} \approx 9.6 \times 10^{-35}$ \\
    \hline
    \end{tabular}
    \caption{\textit{Parameters for four IEEE arithmetic precisions. 
    "Range" denotes the order of magnitude of
    the largest and smallest positive normalized floating point numbers.}}
\end{table}

Many famous algorithms for 
linear systems like gmres, LSE and etc. have been studied in 
mixed precision\cite{doi:10.1137/17M1140819}\cite{doi:10.1137/20M1316822}\cite{doi:10.1137/23M1549079}.
In this paper, we focus on the iterative refinement of the GADI framework\cite{doi:10.1137/21M1450197}.

\subsection{Challenges}

Applying mixed precision to the GADI framework presents several challenges. The reduced numerical range and precision of formats like FP16 can lead to instability or loss of accuracy in iterative computations, particularly when solving ill-conditioned systems or handling sensitive numerical operations. 

FP16 has a limited numerical range, with maximum and minimum representable values significantly smaller than those in FP32 or FP64. When solving large sparse linear systems, certain operations—such as scaling large matrices or intermediate results in iterative steps—may exceed this range, causing overflow. This can lead to incorrect results or instability in the algorithm, particularly when the system matrices have large eigenvalues or poorly conditioned properties. Careful rescaling techniques or mixed-precision strategies are often required to mitigate this issue, ensuring critical computations remain within the numerical limits of FP16.

The reduced precision of FP16 (approximately 3 decimal digits) can introduce rounding errors during iterative processes in GADI. These errors accumulate and may prevent the solver from reaching a sufficiently accurate solution, especially for problems that require high precision or have small residual tolerances. The convergence criteria may need to be relaxed when using FP16, or critical steps (e.g., residual corrections or parameter updates) must be performed in higher precision (FP32 or FP64) to ensure the algorithm converges to an acceptable solution.

In the GADI framework, the Gaussian Process Regression (GPR) method is used to predict the optimal parameter 
$\alpha$ to enhance computational efficiency. However, when applying mixed precision, particularly with FP16, the predicted 
$\alpha$ may fail to ensure convergence. This is because 
$\alpha$ is derived assuming higher precision arithmetic, and the reduced precision and numerical range of FP16 can amplify rounding errors, introduce instability, or exacerbate sensitivity to parameter choices. Consequently, the predicted 
$\alpha$ may no longer balance the splitting matrices effectively, leading to slower convergence or divergence in the mixed precision GADI framework. Addressing this issue may require recalibrating 
$\alpha$ specifically for mixed precision or employing adaptive strategies that dynamically adjust parameters during computation to account for FP16 limitations.

\subsection{Contribution}
To resolve the challenges of mixed-precision GADI, 
we proposed rigorous theoretical analysis and extensive 
empirical validation to ensure that the mixed-precision 
GADI with iterative refinement achieves its goals of 
accelerated computation and high numerical accuracy.
In this article, we introduce a mixed 
precision iterative refinement method of GADI 
as GADI-IR to solve large sparse linear
systems of the form:
\begin{equation}
    Ax=b.
\end{equation}

The goal of this work is to develop a mixed-precision 
iterative refinement method that based on the GADI framework 
using three precisions. Our contributions are summarized as 
follows:
\begin{itemize}
\item We propose a novel mixed-precision iterative refinement 
method, GADI-IR, within the GADI framework, designed for efficiently solving 
large-scale sparse linear systems. We provide a rigorous theoretical 
analysis to ensure the numerical accuracy of the method.

\item We address the challenges of overflow and underflow in 
FP16 arithmetic by applying a regularization parameter $\alpha$ 
to balance the splitting matrices effectively and ensure robust 
convergence in the inner low precision steps.

\item We integrate low-precision arithmetic into the parameter 
prediction process using Gaussian Process Regression (GPR) method and compare 
the performance of GADI-IR with and without the regularization parameter. 
This demonstrates the effectiveness of $\alpha$ in enhancing computational 
efficiency and robustness.

\item By applying GADI-IR to
\begin{itemize}
    \item Three-dimensional convection-diffusion equation,
    \item Continuous-time algebraic Riccati equation (CARE),
    \item Continuous Sylvester equation,
\end{itemize}
we demonstrate the effectiveness 
of the proposed method in solving 
large-scale sparse linear systems. 
The numerical experiments confirm 
that the mixed-precision approach 
significantly accelerates computation 
while maintaining high numerical accuracy. 
The results also highlight the importance 
of the regularization parameter $\alpha$ in 
ensuring robust convergence, particularly 
when using low-precision arithmetic.
\end{itemize}

\subsection{Preliminaries}

We now summarize our notation and our assumptions in this 
article. 

For a non singular matrix $A$ and a vector $x$, we need the normwise condition 
number:
\begin{align}
    \kappa(A)=\|A\|\cdot\|A^{-1}\|.
\end{align}
If the norm is the 2-norm, we denote the condition number as $\kappa_2(A)$:
\begin{align}
    \kappa_2(A)=\|A\|_2\cdot\|A^{-1}\|_2=\sigma_{\max}(A)/\sigma_{\min}(A).
\end{align}

$fl_r(\cdot)$ denotes the evaluation of 
the argument of $fl_r$ in precision $u_r$.

The exact solution of $Ax=b$ is denoted by $x$ and the computed 
solution is denoted by $\hat{x}$.

In algorithm GADI-IR, we use the following notation:
\begin{align*}
    &H=\alpha I + M,\\
    &S=\alpha I + N,\\
    &(2-\omega)\alpha=p.
\end{align*}


\section{Error analysis}
\label{sec:ea}
\subsection{GADI-IR framework}
In \cite{doi:10.1137/21M1450197},  Jiang et al. proposed the GADI framework 
and corresponding algorithm.

Let $M,N\in\mathbb{C}^{n\times n}$ be splitting matrices of $A$ such that: 
$A=M+N$. Given an initial guess $x^{(0)}$, and $\alpha > 0, \omega >0$, 
the GADI framework is:
\begin{align}
\left\{\begin{aligned} {{{}}} & {{} {{} {{( \alpha I+M ) x^{( k+\frac{1} {2} )}=( \alpha I-N ) x^{( k )}+b,}}}} \\ {{{}}} & {{} {{{( \alpha I+N ) x^{( k+1 )}=( N-( 1-\omega) \alpha I ) x^{( k )}+( 2-\omega) \alpha x^{( k+\frac{1} {2} )}.}}}} \\ \end{aligned} \right. 
\label{gadi}
\end{align}
The following theorem, restated from \cite{doi:10.1137/21M1450197}, 
describes the convergence of the GADI framework:
\begin{theorem}[Convergence of the GADI Framework, Jiang et al., 2022]
    The GADI framework(\ref{gadi})
    converges to the unique solution $x$ of the linear system $Ax = b$ for any $\alpha > 0$ and $\omega \in [0, 2)$. Furthermore, the spectral radius $\rho(T(\alpha, \omega))$ satisfies:
    \[
    \rho(T(\alpha, \omega)) < 1,
    \]
    where the iterative matrix $T(\alpha, \omega)$ is defined as:
    \[
    T(\alpha, \omega) = (\alpha I + N)^{-1} (\alpha I + M)^{-1} (\alpha^2 I + MN - (1-\omega)\alpha A).
    \]
    \label{thm1}
\end{theorem}

The GADI
framework for solving large sparse linear systems, 
including its full-precision error analysis and 
convergence properties have been thoroughly 
investigated\cite{doi:10.1137/21M1450197}. 
It\cite{doi:10.1137/21M1450197} unifies 
existing ADI methods and introduces new schemes, 
while addressing the critical issue 
of parameter selection by employing Gaussian Process Regression (GPR) method 
for efficient prediction. Numerical 
results demonstrate that the GADI 
framework significantly improves computational 
performance and scalability, solving much larger 
systems than traditional methods while maintaining accuracy.

Building on this foundation\cite{doi:10.1137/21M1450197}, 
we designed a mixed-precision GADI algorithm GADI-IR 
that strategically combines low-precision 
and high-precision computations. Computationally 
intensive yet numerically stable operations, 
such as matrix-vector multiplications, are 
performed in FP16 to leverage its speed and 
efficiency, while critical steps like residual 
corrections, parameter updates, and convergence 
checks are handled in higher precision to ensure robustness.

\begin{algorithm}
    \begin{algorithmic}
        \Require: $\alpha,\omega,\xi,H,S,\varepsilon,\hat{r}^{(0)},k=0,\hat{x}^{(0)}=1$
        \While{$\|\hat{r}^{(k)}\|_2^2>\|\hat{r}^{(0)}\|_2^2\xi$}
        \State{Step 1: $\hat{r}^{(k)}=b-A\hat{x}^{(k)}$}\Comment{$u_f$}
        \State{Step 2: Solve $H\hat{z}^{(k)}\approx\hat{r}^{(k)}$ such that $tol\leq \varepsilon\|\hat{r}^{(0)}\|_2$}\Comment{$u_r$}
        \State{Step 3: Solve $S\hat{y}^{(k)}\approx(2-\omega)\alpha \hat{z}^{(k)}$ such that $tol\leq \varepsilon\|\hat{r}^{(0)}\|_2$}\Comment{$u_r$}
        \State{Step 4: Compute $\hat{x}^{(k+1)}=\hat{x}^{(k)}+\hat{y}^{(k)}$}\Comment{$u$}
        \State{Step 5: Compute $k=k+1$}
        \EndWhile
    \end{algorithmic}
    \caption{GADI-IR}
    \label{alg1}
\end{algorithm}

We present the rounding error analysis of Algorithm \ref{alg1} in the following sections, 
which include forward error bounds and backward error bounds in section~\ref{sec:ea}.  The significance 
of regularization to avoid underflow
and overflow when half precision is used is explained in section~\ref{sec:pr}. 
In this section we also 
specialize the results of Gaussian Process Regression (GPR) prediction to the GADI-IR algorithm and compare it with 
the regularization method.
Numerical experiments presented in
section~\ref{sec:ne} confirm the predictions of the analysis. Conclusions are given in 
section~\ref{sec:con}.

\subsection{Forward analysis}
    Let step 2 and step 3 in Algorithm \ref{alg1} 
    performed using a backward stable algorithm, then 
    there exists $G_k$ and $F_k$ such that:
    \begin{align}
        &(H+G_k)\hat{z}^{(k)}=\hat{r}^{(k)},\\
        &(S+F_k)\hat{y}^{(k)}=(2-\omega)\alpha\hat{z}^{(k)},
    \end{align}
    where:
    $$
    \|G_k\|\leq\phi(n)u_r\|M\|,
    $$
    $$
    \|F_k\|\leq\varphi(n)u_r\|N\|.
    $$
    where $\phi(n)$ and $\varphi(n)$ are reasonably small 
    functions of matrix size $n$.

Considering the computation of $\hat{r}^{(k)}$, There are two stages. 
First, $\hat{s}^{(k)} = fl_f(b - 
A\hat{x}^{(k)})$ is formed in precision $u_f$, so that:
\begin{align}
    \|\hat{s}^{(k)}\|\leq\varphi_1(n)u_f(\|A\|\|\hat{x}^{(k)}\|+\|b\|).
\end{align}
Second, the residual is rounded to precision $u_r$, so 
$\hat{r}^{(k)}=fl_r(\hat{s}^{(k)})=\hat{s}^{(k)}+f_k$. Hence:
\begin{align}
    \hat{r}^{(k)}=b-A\hat{x}^{(k)}+\Delta \hat{r}^{(k)},
\end{align}
where:
$$
\|\Delta \hat{r}^{(k)}\|\leq u_r\|b-A\hat{x}^{(k)}\|+
(1+u_r)u_f(\|A\|\|\hat{x}^{(k)}\|+\|b\|).
$$

So the classical error bounds in GADI-IR 
in steps 1 and 4 
are hold:
\begin{align}
    &\hat{r}^{(k)}=fl_f(b-A\hat{x}^{(k)})=b-A\hat{x}^{(k)}+\Delta \hat{r}^{(k)},\\
    &\hat{x}^{(k+1)}=fl(\hat{x}^{(k)}+\hat{y}^{(k)})=\hat{x}^{(k)}+\hat{y}^{(k)}+\Delta \hat{x}^{(k)},
\end{align}
where:
\begin{align}
\|\Delta \hat{r}^{(k)}\|&\leq u_r\|b-A\hat{x}^{(k)}\|+
(1+u_r)u_f(\|A\|\|\hat{x}^{(k)}\|+\|b\|),\\
\|\Delta \hat{x}^{(k)}\|&\leq \varphi_2(n)
u(\|\hat{x}^{(k)}\|+\|\hat{y}^{(k)}\|).
\end{align}

\begin{lemma}[\cite{10.5555/500666}]
    If $\phi(n)\kappa(S)u_r<1/2$, then $(S+F_k)$ is 
    non singular and 
    $$
    (S+F_k)^{-1}=(I+J_k)S^{-1}$$
    where$$\|J_k\|\leq\frac{\varphi(n)\kappa(S)u_r}
    {1-\phi(n)\kappa(S)u_r}<1.
    $$
    \label{lem1}
\end{lemma}

Then we can prove the following theorem.
\begin{theorem}
    Let Algorithm \ref{alg1} be applied to the linear system $A x=b$ , where $A \in\mathbb{R}^{n \times n}$ is nonsingular, 
    and assume the solver used in step 2 and 3 is backward stable. For $k \geq0$ the computed iterate $\hat{x}_{k+1}$ satisfies
    $$
    \|x-\hat{x}_{k+1}\|\leq \alpha_F \|x-\hat{x}_k\|+\beta_F\|x\|,
    $$
    where
    \begin{align*}
        \alpha_F&=\lambda+(2-\omega)\theta\hat{\kappa}(HS)+\varphi_2(n)u+4(2-\omega)\varphi_2(n)\hat{\kappa}(HS)u(1+u_r)(1+u_f)\notag\\
        &+4(2-\omega)\hat{\kappa}(HS)u_r+4(2-\omega)(1+u_r)u_f\hat{\kappa}(HS), \\
        \beta_F&=\varphi_2(n)u+4(2-\omega)\varphi_2(n)\hat{\kappa}(HS)u(1+u_r)u_f+8(2-\omega)(1+u_r)u_f\hat{\kappa}(HS),
    \end{align*}
    and
    \begin{align*}
        \hat{\kappa}(HS)=\kappa(H)\kappa(S)\frac{\alpha\|H\|+\alpha\|S\|+2\alpha^2}{\|H\|\|S\|}.
    \end{align*}
    \label{thm2}
\end{theorem}
\begin{proof}
    First, the error between the exact solution $x$ and the $k$th iterative 
    solution $\hat{x}^{(k+1)}$ need to be estimated. 
    From equation (2.7) it comes:
    \begin{align*}
        x-\hat{x}^{(k+1)}=x-\hat{x}^{(k)}-\hat{y}^{(k)}-\Delta \hat{x}^{(k)},
    \end{align*}
    and then using equation (2.2), (2.3) and (2.5), the 
    error between the exact solution $x$ and the $k$th iterative 
    solution $\hat{x}^{(k+1)}$ can be represented as the following equation:
    \begin{align*}
        &x-\hat{x}^{(k+1)}\\
        =&x-\hat{x}^{(k)}-\hat{y}^{(k)}-\Delta \hat{x}^{(k)}\\
        =&x-\hat{x}^{(k)}-(S+F_k)^{-1}(2-\omega)\alpha \hat{z}^{(k)}-\Delta \hat{x}^{(k)}\\
        =&x-\hat{x}^{(k)}-(S+F_k)^{-1}(H+G_k)^{-1}(2-\omega)\alpha \hat{r}^{(k)}-\Delta \hat{x}^{(k)}\\
        =&x-\hat{x}^{(k)}-(I+J_k)S^{-1}(I+L_k)H^{-1}p\hat{r}^{(k)}-\Delta \hat{x}^{(k)}\\
        =&x-\hat{x}^{(k)}-(I+J_k)S^{-1}(I+L_k)H^{-1}p(b-A\hat{x}^{(k)}+\Delta\hat{r}^{(k)})\\
        &-\Delta \hat{x}^{(k)}\\
        =&x-\hat{x}^{(k)}-\Delta \hat{x}^{(k)}\\
        &-(I+J_k)S^{-1}(I+L_k)H^{-1}Ap(x-\hat{x}^{(k)}+A^{-1}\Delta\hat{r}^{(k)})\\
        =&x-\hat{x}^{(k)}-\Delta \hat{x}^{(k)}-(HS)^{-1}Ap(x-\hat{x}^{(k)})\\
        &-S^{-1}L_kH^{-1}Ap(x-\hat{x}^{(k)})-J_kS^{-1}H^{-1}Ap(x-\hat{x}^{(k)})\\
        &-J_kS^{-1}L_kH^{-1}Ap(x-\hat{x}^{(k)})\\
        &-(I+J_k)S^{-1}(I+L_k)H^{-1}p\Delta\hat{r}^{(k)}.
    \end{align*}

    Considering the iterative matrix $T(\alpha,\omega)$ of GADI framework in theorem \ref{thm1}, 
    the following equation can be obtained: 
    \begin{align*}
        x-\hat{x}^{(k+1)}=&T(\alpha,\omega)(x-\hat{x}^{(k)})-\Delta \hat{x}^{(k)}\\
        &-S^{-1}L_kH^{-1}Ap(x-\hat{x}^{(k)})-J_kS^{-1}H^{-1}Ap(x-\hat{x}^{(k)})\\
        &-J_kS^{-1}L_kH^{-1}Ap(x-\hat{x}^{(k)})\\
        &-(I+J_k)S^{-1}(I+L_k)H^{-1}p\Delta\hat{r}^{(k)}.
    \end{align*}

    Taking the norms of both sides of the last equation and using the fact that $\|J_k\| < 1$ 
    in lemma \ref{lem1}, we have the following inequality for the norm error 
    of the exact solution and the $k$th iterative solution:
    \begin{align*}
        &\|x-\hat{x}^{(k+1)}\|\\
        \leq&\|T(\alpha,\omega)(x-\hat{x}^{(k)})\|+\|\Delta \hat{x}^{(k)}\|\\
        &+p\|L_k\|\|H^{-1}\|\|S^{-1}\|\|A\|\|x-\hat{x}^{(k)}\|\\
        &+p\|J_k\|\|S^{-1}\|\|H^{-1}\|\|A\|\|x-\hat{x}^{(k)}\|\\
        &+p\|J_k\|\|L_k\|\|H^{-1}\|\|S^{-1}\|\|A\|\|x-\hat{x}^{(k)}\|\\
        &+p\|H^{-1}\|\|S^{-1}\|\|I+J_k\|\|I+L_k\|\|\Delta\hat{r}^{(k)}\|\\
        \leq&\|T(\alpha,\omega)(x-\hat{x}^{(k)})\|+\|\Delta \hat{x}^{(k)}\|\\
        &+p\|H^{-1}\|\|S^{-1}\|\|A\|(\|J_k\|+\|L_k\|+\|J_k\|\|L_k\|)\|x-\hat{x}^{(k)}\|\\
        &+4p\|H^{-1}\|\|S^{-1}\|\|\Delta\hat{r}^{(k)}\|.
    \end{align*}

    To make the equation clear, let $\|J_k\|+\|L_k\|+\|J_k\|\|L_k\|=\theta$ 
    and use equation (2.8) and (2.9), the inequality can be simplified as:
    \begin{equation}
        \begin{aligned}
            &\|x-\hat{x}^{(k+1)}\|\\
            \leq&\|T(\alpha,\omega)(x-\hat{x}^{(k)})\|+p\theta\|H^{-1}\|\|S^{-1}\|\|A\|\|x-\hat{x}^{(k)}\|\\
            &+\varphi_2(n)
            u(\|\hat{x}^{(k)}\|+\|\hat{y}^{(k)}\|)\\
            &+4p\|H^{-1}\|\|S^{-1}\|(u_r\|A\|\|x-\hat{x}^{(k)}\|\\
            &+(1+u_r)u_f(\|A\|\|\hat{x}^{(k)}\|+\|b\|)).
        \end{aligned}
    \end{equation}

    So we have the norm error estimating formula between the exact solution $x$ and the $k$-th iterative 
    solution $\hat{x}^{(k+1)}$. Next, it is necessary to estimate the 
    norm error of the $k$th $\|\hat{y}^{(k)}\|$.

    Using the fact that $A$ can be splitted as $A=M+N$, and $H=\alpha I +M,S=\alpha I+N$, 
    then 
    the norm inequality exists:
    \begin{equation}
        \begin{aligned}
            \alpha\|H^{-1}\|\|S^{-1}\|\|A\|&\leq\alpha\|H^{-1}\|\|S^{-1}\|\|H+S-2\alpha I\|\\
            &\leq\alpha\|H^{-1}\|\|S^{-1}\|(\|H\|+\|S\|+2|\alpha|)\\
            &\leq\kappa(H)\kappa(S)\frac{\alpha\|H\|+\alpha\|S\|+2\alpha^2}{\|H\|\|S\|}.
        \end{aligned}
    \end{equation}
    Again, triangle inequality yields:
    \begin{equation}
        \begin{aligned}
            \|\hat{x}^{(k)}\|&\leq\|x\|+\|\hat{x}^{(k)}-x\|,
        \end{aligned}
    \end{equation}
    then, combining (2.13) with the fact that $Ax = b$, the following inequality can be derived:
    \begin{equation}
        \begin{aligned}
            \|A\|\|\hat{x}^{(k)}\|+\|b\|&\leq\|A\|\|x-\hat{x}^{(k)}\|+2\|A\|\|x\|.
        \end{aligned}
    \end{equation}
    Using (2.6) and (2.8), the estimate of $\|\hat{r}^{(k)}\|$ is:
    \begin{align*}
        \|\hat{r}^{(k)}\|&\leq\|b-A\hat{x}^{(k)}\|+\|\Delta\hat{r}^{(k)}\|\notag\\
        &\leq(1+u_r)\|A\|\|x-\hat{x}^{(k)}\|\notag\\
        &+(1+u_r)u_f(\|A\|\|\hat{x}^{(k)}\|+\|b\|).
    \end{align*}
    Combining the above inequalities (2.13),(2.14) and (2.15) 
    with lemma \ref{lem1}, there exists the estimate of $\|\hat{y}^{(k)}\|$:

    \begin{equation}
        \begin{aligned}
            \|\hat{y}^{(k)}\|=&p\|(S+F_k)^{-1}\|\|\hat{z}^{(k)}\|\\
            =&p\|(S+F_k)^{-1}(H+G_k)^{-1}\hat{r}^{(k)}\|\\
            =&p\|(I+J_k)S^{-1}(I+L_k)H^{-1}\hat{r}^{(k)}\|\\
            \leq& 4p\|H^{-1}\|\|S^{-1}\|\|\hat{r}^{(k)}\|\\
            \leq& 4p\|H^{-1}\|\|S^{-1}\|(\|A\|\|x-\hat{x}^{(k)}\|+u_r\|A\|\|x-\hat{x}^{(k)}\|\\
            &+(1+u_r)u_f(\|A\|\|\hat{x}^{(k)}\|+\|b\|))\\
            \leq& 4p(1+u_r)\|H^{-1}\|\|S^{-1}\|(\|A\|\|x-\hat{x}^{(k)}\|)\\
            &+4p(1+u_r)u_f\|H^{-1}\|\|S^{-1}\|(\|A\|\|\hat{x}^{(k)}\|+\|b\|)\\
            \leq& 4p(1+u_r)\|H^{-1}\|\|S^{-1}\|(\|A\|\|x-\hat{x}^{(k)}\|)\\
            &+4p(1+u_r)u_f\|H^{-1}\|\|S^{-1}\|(\|A\|\|x-\hat{x}^{(k)}\|+2\|A\|\|x\|)\\
            =& 4p(1+u_r)(1+u_f)\|H^{-1}\|\|S^{-1}\|\|A\|\|x-\hat{x}^{(k)}\|\\
            &+8p(1+u_r)u_f\|H^{-1}\|\|S^{-1}\|\|A\|\|x\|.
        \end{aligned}
    \end{equation}
    
    Theorem \ref{thm1} states that the radius of the iterative matrix $T(\alpha,\omega)$ 
    is $\rho(T(\alpha,\omega))$ and 
    $\rho(T(\alpha,\omega))<1$, so the following inequality can be derived: 
    \begin{align*}
        \|T(\alpha,\omega)(x-\hat{x}^{(k)})\|\leq\lambda\|x-\hat{x}^{(k)}\|,
    \end{align*}
    where $\lambda<1$.
    Injecting equations (2.12), (2.13), (2.14), (2.15) and (2.16) in equation (2.11) and 
    let $\kappa(H)\kappa(S)\frac{\alpha\|H\|+\alpha\|S\|+2\alpha^2}{\|H\|\|S\|}=\hat{\kappa}(HS)$ yields:
    \begin{equation}
        \begin{aligned}
            &\|x-\hat{x}^{(k+1)}\|\\
            \leq&\|T(\alpha,\omega)(x-\hat{x}^{(k)})\|+(2-\omega)\theta\hat{\kappa}(HS)\|x-\hat{x}^{(k)}\|+\varphi_2(n)u\|x\|
            +\varphi_2(n)u\|x-\hat{x}^{(k)}\|\\
            &+4(2-\omega)\varphi_2(n)\hat{\kappa}(HS)u(1+u_r)((1+u_f)\|x-\hat{x}^{(k)}\|
            +2u_f\|x\|)\\
            &+4(2-\omega)\hat{\kappa}(HS)u_r\|x-\hat{x}^{(k)}\|+
            4(2-\omega)(1+u_r)u_f\hat{\kappa}(HS)(\|x-\hat{x}^{(k)}\|+2\|x\|)\\
            \leq&(\lambda+(2-\omega)\theta\hat{\kappa}(HS)+\varphi_2(n)u+
            4(2-\omega)\varphi_2(n)\hat{\kappa}(HS)u(1+u_r)(1+u_f)\\
            &+4(2-\omega)\hat{\kappa}(HS)u_r+
            4(2-\omega)(1+u_r)u_f\hat{\kappa}(HS))\|x-\hat{x}^k\|\\
            &+(\varphi_2(n)u+4(2-\omega)\varphi_2(n)\hat{\kappa}(HS)u(1+u_r)u_f+
            8(2-\omega)(1+u_r)u_f\hat{\kappa}(HS))\|x\|\\
            =&\alpha_F\|x-\hat{x}^{(k)}\|+\beta_F\|x\|,
        \end{aligned}
    \end{equation}
    where:
    \begin{align*}
        \alpha_F&=\lambda+(2-\omega)\theta\hat{\kappa}(HS)+
        \varphi_2(n)u+4(2-\omega)\varphi_2(n)\hat{\kappa}(HS)u(1+u_r)(1+u_f)\notag\\
        &+4(2-\omega)\hat{\kappa}(HS)u_r+4(2-\omega)(1+u_r)u_f\hat{\kappa}(HS), \\
        \beta_F&=\varphi_2(n)u+4(2-\omega)\varphi_2(n)\hat{\kappa}(HS)u(1+u_r)u_f+
        8(2-\omega)(1+u_r)u_f\hat{\kappa}(HS).
    \end{align*}
\end{proof}

According to theorem \ref{thm2}, we have the following corollary that 
provides the error estimation of GADI-IR.
\begin{corollary}
Let $x$ be the exact solution of $Ax=b$ and $x^*$ be the solution 
calculated by GADI-IR, then we have the 
following error estimation: 
\begin{equation}
    \begin{aligned}
        \lim_{k\rightarrow\infty}\|x-\hat{x}_{k}\|&=\|x-x^*\|\\
        &\leq\beta_F(1-\alpha_F)^{-1}\|x\|\\
        &=\frac{\phi_F(n)\hat{\kappa}(HS)u}{1-\psi_F(n)\hat{\kappa}(HS)u_r}\|x\|.
    \end{aligned}
\end{equation}
\label{cor1}
\end{corollary}

\begin{proof}

To make sure that the result to converge to the exact solution, it is necessary 
that 
    \begin{align*}
        \alpha_F <1, \\
        \beta_F <1.
    \end{align*}

Using theorem \ref{thm2}, we have
\begin{align*}
&\lim_{k\rightarrow\infty}\|x-\hat{x}_{k+1}\|\\
&\leq\lim_{k\rightarrow\infty}(\alpha_F\|x-\hat{x}_k\|+\beta_F\|x\|)\\
&\leq\lim_{k\rightarrow\infty}(\alpha_F^k\|x-\hat{x}_1\|+\beta_F\frac{1-\alpha_F^k}{1-\alpha_F}\|x\|).
\end{align*}

Note that $\alpha_F$ and $\beta_F$ are of the form:
$$
\alpha_{F}=\psi_{F} ( n ) \hat{\kappa}(HS) u_{r}, \beta_{F}=\varphi_2(n)u+\phi_{F} ( n ) \hat{\kappa}(HS) u_f, 
$$
are respectively determined by $u_r$ and $u$, so ${x_k}$ converges. Now we set 
$$\lim_{k\rightarrow\infty}\hat{x}_{k}=x^*.$$

Combining (2.15) and using the form of $\alpha_F$ and $\beta_F$, we have
\begin{equation}
    \begin{aligned}
        \lim_{k\rightarrow\infty}\|x-\hat{x}_{k}\|&=\|x-x^*\|\\
        &\leq\beta_F(1-\alpha_F)^{-1}\|x\|\\
        &=\frac{\varphi_2(n)u+\phi_{F} ( n ) \hat{\kappa}(HS) u_f}{1-\psi_F(n)\hat{\kappa}(HS)u_r}\|x\|.
    \end{aligned}
\end{equation}
\end{proof}

From corollary \ref{cor1}, it can be seen that 
the term $\alpha_F$ is the rate of convergence and depends 
on the condition number of the
matrix $H,S$, and the precision used $u_r$. The term $\beta_F$ is the limiting accuracy
of the method and depends on the precision accuracy used $u$.

\subsection{Backward analysis}
\begin{lemma}[\cite{10.5555/500666}]
    If $\mu(n)\kappa(S)u_r<1/2$, then $(H+F_k)$ is 
    non singular and 
    \begin{equation}
      \begin{aligned}
          (S+F_k)^{-1}=S^{-1}(I+P_k),
      \end{aligned}
  \end{equation}
  where:
  \begin{equation}
      \begin{aligned}
          \|P_k\|\leq\frac{\mu(n)\kappa(S)u_r}{1-\mu(n)\kappa(S)u_r}\leq1.
      \end{aligned}
  \end{equation}
  \label{lem2}
\end{lemma}

\begin{theorem}
    Let Algorithm \ref{alg1} be applied to a linear system $Ax = b$ with a nonsingular matrix $A\in\mathbb{R}^{n\times n}$  and assume the solver used in step 2 and 3 is backward stable. 
    Then for $k\geq 0$ the computed iterate $\hat{x}_{k+ 1}$ satisfies
\begin{equation}
    \begin{aligned}
        \frac{\|b-Ax_{k+1}\|}{\|A\|\|x_{k+1}\|}\leq\alpha_B\frac{\|b-Ax_{k}\|}{\|A\|\|x_{k}\|}+\beta_B,
    \end{aligned}
\end{equation}
where:
\begin{align*}
    \alpha_B=&\gamma(\lambda+(2-\omega)\eta\hat{\kappa}(HS)+
    4(2-\omega)\hat{\kappa}(HS)u_r+4(2-\omega)\hat{\kappa}(HS)(1+u_r)u_f),\\
    \beta_B=&8(2-\omega)\gamma\hat{\kappa}(HS)(1+u_r)u_f+\varphi_2(n)\gamma u\notag\\
    &+\varphi_2(n)(1-\varphi_2(n)u)^{-1}u(1+(1+\varphi_2(n)u)\gamma).
\end{align*}
and
\begin{align*}
    \hat{\kappa}(HS)&=\kappa(H)\kappa(S)\frac{\alpha\|H\|+\alpha\|S\|+2\alpha^2}{\|H\|\|S\|}.
\end{align*}
\label{thm3}
\end{theorem}

\begin{proof}

Building upon equations (2.2) and (2.3), we can derive the following sequence of equations:
\begin{align*}
    &x-\hat{x}^{(k+1)}\\
    =&x-\hat{x}^{(k)}-\hat{y}^{(k)}-\Delta \hat{x}^{(k)}\\
    =&x-\hat{x}^{(k)}-(S+F_k)^{-1}p \hat{z}^{(k)}-\Delta \hat{x}^{(k)}\\
    =&x-\hat{x}^{(k)}-p(S+F_k)^{-1}(H+G_k)^{-1}\hat{r}^{(k)}-\Delta \hat{x}^{(k)}.
\end{align*}
Subsequently, by applying equation (2.5) for the residual computation 
and leveraging the key result from lemma \ref{lem2} 
regarding the inverse matrix structure, we can derive:
\begin{align*}
    &x-\hat{x}^{(k+1)}\\
    =&x-\hat{x}^{(k)}-pS^{-1}(I+P_k)H^{-1}(I+Q_k)\hat{r}^{(k)}-\Delta \hat{x}^{(k)}\\
    =&x-\hat{x}^{(k)}-\Delta \hat{x}^{(k)}\\
    &-pS^{-1}(I+P_k)H^{-1}(I+Q_k)(b-A\hat{x}^{(k)}+\Delta\hat{r}^{(k)}).
\end{align*}
To obtain an expression for the residual, 
we multiply both sides of the equation by the 
coefficient matrix $A$ on the left, which yields:
\begin{align*}
    &b-A\hat{x}^{(k+1)}\\
    =&b-A\hat{x}^{(k)}-A\Delta\hat{x}^{(k)}\\
    &-pAS^{-1}(I+P_k)H^{-1}(I+Q_k)(b-A\hat{x}^{(k)}+\Delta\hat{r}^{(k)})\\
    =&b-A\hat{x}^{(k)}-pAS^{-1}(I+P_k)H^{-1}(I+Q_k)(b-A\hat{x}^{(k)})\\
    &-pAS^{-1}(I+P_k)H^{-1}(I+Q_k)\Delta\hat{r}^{(k)}-A\Delta\hat{x}^{(k)}\\
    =&T(\alpha,\omega)(b-A\hat{x}^{(k)})-A\Delta \hat{x}^{(k)}\\
    &-S^{-1}P_kH^{-1}Ap(b-A\hat{x}^{(k)})-Q_kS^{-1}H^{-1}Ap(b-A\hat{x}^{(k)})\\
    &-Q_kS^{-1}P_kH^{-1}Ap(b-A\hat{x}^{(k)})\\
    &-A(I+Q_k)S^{-1}(I+P_k)H^{-1}p\Delta\hat{r}^{(k)}.
\end{align*}

Taking the norm of both sides and using the fact that $\|P_k\| < 1$ and 
letting $\|P_k\|+\|Q_k\|+\|P_k\|\|Q_k\|=\eta$ again gives:
\begin{align*}
    &\|b-A\hat{x}^{(k+1)}\|\\
    \leq&\|T(\alpha,\omega)(b-A\hat{x}^{(k)})\|\\
    &+p\|S^{-1}\|\|P_k\|\|H^{-1}\|\|A\|\|b-A\hat{x}^{(k)}\|\\
    &+p\|Q_k\|\|S^{-1}\|\|H^{-1}\|\|A\|\|b-A\hat{x}^{(k)}\|\\
    &+p\|Q_k\|\|S^{-1}\|\|P_k\|\|H^{-1}\|\|A\|\|b-A\hat{x}^{(k)}\|\\
    &+p\|A\|\|S^{-1}\|\|H^{-1}\|\|I+P_k\|\|I+Q_k\|\|\Delta\hat{r}^{(k)}\|\\
    &+\|A\|\|\Delta\hat{x}^{(k)}\|\\
    \leq&\|T(\alpha,\omega)(b-A\hat{x}^{(k)})\|+p\eta\|H^{-1}\|\|S^{-1}\|\|b-A\hat{x}^{(k)}\|\\
    &+4p\|H^{-1}\|\|S^{-1}\|\|A\|\|\Delta\hat{r}^{(k)}\|+
    \|A\|\|\Delta\hat{x}^{(k)}\|.
\end{align*}

Applying Equations (2.8) and (2.9) to further 
refine our analysis, we obtain the following expression:
\begin{equation}
    \begin{aligned}
        &\|b-A\hat{x}^{(k+1)}\|\\
        \leq&\|T(\alpha,\omega)(b-A\hat{x}^{(k)})\|+p\eta\|H^{-1}\|\|S^{-1}\|\|A\|\|b-A\hat{x}^{(k)}\|\\
            &+4p\|H^{-1}\|\|S^{-1}\|\|A\|(u_r\|b-A\hat{x}^{(k)}\|\\
            &+(1+u_r)u_f(\|A\|\|\hat{x}^{(k)}\|+\|b\|))\\
            &+\varphi_2(n)
            u\|A\|(\|\hat{x}^{(k)}\|+\|\hat{y}^{(k)}\|).
    \end{aligned}
\end{equation}
For $\|b\|$ according to $b=b-A\hat{x}^{(k)}+A\hat{x}^{(k)}$ we have:
\begin{equation}
    \|b\|\leq\|b-A\hat{x}^{(k)}\|+\|A\|\|\hat{x}^{(k)}\|.
\end{equation}

Then we need to calculate the $\|\hat{y}^{k}\|$ in equation (2.20), using equation (2.7) 
and (2.9):
\begin{align*}
    \|\hat{y}^{(k)}\|=&\|\hat{x}^{(k+1)}-\hat{x}^{(k)}-\Delta \hat{x}^{(k)}\|\notag\\
    \leq&\|\hat{x}^{(k+1)}\|+\|\hat{x}^{(k)}\|+\|\Delta \hat{x}^{(k)}\|\notag\\
    \leq&\|\hat{x}^{(k+1)}\|+\|\hat{x}^{(k)}\|+\varphi_2(n)u\|\hat{x}^{(k)}\|
    +\varphi_2(n)u\|\hat{y}^{(k)}\|\notag
\end{align*}
then:
\begin{align}
    \|\hat{y}^{(k)}\|\leq(1-\varphi_2(n)u)^{-1}(\|\hat{x}^{(k+1)}\|+(1+\varphi_2(n)u)\|\hat{x}^{(k)}\|).
\end{align}

Finally, injecting equations (2.11), (2.21) and (2.22) in equation (2.20) yields:
\begin{align}
    \|b-Ax_{k+1}\|\leq&\lambda\|b-Ax_{k}\|+(2-\omega)\eta\hat{\kappa}(HS)\|b-A\hat{x}^{(k)}\|\notag\\
    &+4(2-\omega)\hat{\kappa}(HS)u_r\|b-A\hat{x}^{(k)}\|\notag\\
    &+4(2-\omega)\hat{\kappa}(HS)(1+u_r)u_f\|A\|\|\hat{x}^{(k)}\|\notag\\
    &+4(2-\omega)\hat{\kappa}(HS)(1+u_r)u_f(\|b-A\hat{x}^{(k)}\|+\|A\|\|\hat{x}^{(k)}\|)\notag\\
    &+\varphi_2(n)u\|A\|\|\hat{x}^{(k)}\|\notag\\
    &+\varphi_2(n)(1-\varphi_2(n)u)^{-1}u\|A\|(\|\hat{x}^{(k+1)}\|+(1+\varphi_2(n)u)\|\hat{x}^{(k)}\|).\notag
\end{align}
From forward analysis note that there exists $\gamma$ 
so that $\|\hat{x}^{(k)}\|\leq\gamma\|\hat{x}^{(k+1)}\|$, then:
\begin{align*}
    &\frac{\|b-Ax_{k+1}\|}{\|A\|\|x_{k+1}\|}\\
    \leq&\gamma(\lambda+(2-\omega)\eta\hat{\kappa}(HS)+
    4(2-\omega)\hat{\kappa}(HS)u_r+4(2-\omega)\hat{\kappa}(HS)(1+u_r)u_f)\frac{\|b-Ax_{k}\|}{\|A\|\|x_{k}\|}\notag\\
    &+8(2-\omega)\gamma\hat{\kappa}(HS)(1+u_r)u_f+\varphi_2(n)\gamma u\notag\\
    &+\varphi_2(n)(1-\varphi_2(n)u)^{-1}u(1+(1+\varphi_2(n)u)\gamma)\notag\\
    =&\alpha_B\frac{\|b-Ax_{k}\|}{\|A\|\|x_{k}\|}+\beta_B,\notag
\end{align*}

where:
\begin{align*}
    \alpha_B=&\gamma(\lambda+(2-\omega)\eta\hat{\kappa}(HS)+
    4(2-\omega)\hat{\kappa}(HS)u_r+4(2-\omega)\hat{\kappa}(HS)(1+u_r)u_f),\\
    \beta_B=&8(2-\omega)\gamma\hat{\kappa}(HS)(1+u_r)u_f+\varphi_2(n)\gamma u\notag\\
    &+\varphi_2(n)(1-\varphi_2(n)u)^{-1}u(1+(1+\varphi_2(n)u)\gamma).\notag
\end{align*}
\end{proof}

It can be seen that 
the term $\alpha_B$ is the rate of convergence and depends 
on the condition number of the
matrix $H,S$ and parameter $\alpha$, and the precision used $u_r$. The term $\beta_B$ is the limiting accuracy
of the method and depends on the precision accuracy used $u$.

\section{Parameter prediction and regularization}
\label{sec:pr}
\subsection{Regularization}
\label{subsec:reg}
The regularization parameter $\alpha$ plays a crucial role in determining the
performance and stability of GADI. In this section, we analyze how to optimally select
the regularization parameter $\alpha$ to effectively balance the splitting matrices
and ensure robust convergence in the inner low-precision steps of GADI-IR. We focus
particularly on its impact on numerical stability and convergence behavior when
operating in mixed-precision environments.

In Algorithm \ref{alg1}, inner loop step 2 and step 3 are performed 
with coefficient matrix $H=\alpha I +H, S=\alpha I+S$ which are of the form:
\begin{align}
    \alpha I + U,
\end{align}
with regularization parameter $\alpha$. 

For the regularized matrix in (3.1), we can explicitly compute its condition number. The 2-norm condition number of matrix $\alpha I + U$ is given by:
\begin{equation}
    \begin{aligned}
        \kappa_2(\alpha I + U) &= \frac{\sigma_{\max}(\alpha I + U)}{\sigma_{\min}(\alpha I + U)}\\
        &= \frac{\alpha + \sigma_{\max}(U)}{\alpha + \sigma_{\min}(U)},
    \end{aligned}
\end{equation}
where $\sigma_{\max}$ and $\sigma_{\min}$ denote the 
largest and smallest singular values respectively. 
This expression reveals an important property: 
as the regularization parameter $\alpha$ increases, 
the ratio between the maximum and minimum singular 
values decreases, thereby improving the condition 
number of the regularized matrix $\alpha I + U$.

Then, we can analyse the $\hat{\kappa}_2(HS)$, for $\hat{\kappa}_2(HS)$ we have:
\begin{equation}
    \begin{aligned}
        \hat{\kappa}_2(HS)
        =&\kappa_2(H)\kappa_2(S)\frac{\alpha\|H\|_2+\alpha\|S\|_2+2\alpha^2}{\|H\|_2\|S\|_2}\\
        =&\kappa_2(H)\kappa_2(S)\frac{\alpha\|\alpha I+N\|_2+
        \alpha\|\alpha I+M\|_2+2\alpha^2}{\|\alpha I+M\|_2\|\alpha I+N\|_2}\\
        =&\frac{\alpha(\alpha+\sigma_{\max}(M)+\alpha+\sigma_{\max}(N))+\alpha^2}
        {(\alpha+\sigma_{\min}(M))(\alpha+\sigma_{\min}(N))}\\
        =&\frac{4\alpha^2+\alpha(\sigma_{\max}(M)+\sigma_{\max}(N))}
        {\alpha^2+\alpha(\sigma_{\min}(M)+\sigma_{\min}(N))+\sigma_{\min}(M)\sigma_{\min}(N)}.
    \end{aligned}
\end{equation}
It is obvious from (3.3) that $\hat{\kappa}_2(HS)$ is a monotonically decreasing function 
with respect to $\alpha$.
Also, it is clearly from (3.3) that:
\begin{align*}
    \lim_{\alpha \to \infty} \hat{\kappa}_2(HS) = 4.
\end{align*}
This sensitivity to $\alpha$ becomes particularly 
pronounced in mixed-precision environments, where 
reduced precision operations during iteration can 
amplify small errors, especially if the regularization 
term is not well-calibrated. Therefore, determining 
an optimal value for $\alpha$ is crucial for 
maintaining the robustness of the mixed-precision 
GADI-IR algorithm, as it ensures that the computational 
efficiency gains are achieved without compromising 
solution accuracy or convergence reliability.

\subsection{Backward analysis for regularization}
\label{subsec:breg}
In this section, 2-Norm will be used as the symbolic norm to 
satisfy (3.2).
Theorem \ref{thm3} provides the backward error analysis of GADI-IR, where 
$\alpha_B$ determines the convergence rate and $\beta_B$ characterizes 
the ultimate achievable accuracy of the method.

According to \cite{doi:10.1137/23m1557209}, 
for a given matrix $U$, reducing 
its precision can lead to an improvement in 
its condition number. 
When down-casting a matrix $U$ 
(e.g., from double precision to single precision), 
its smallest singular value increases  
while the largest singular value remains largely unchanged 
which can be expressed mathematically as:
\begin{align}
    \kappa_2(U_r)\leq \kappa_2(U),
\end{align}
where $U_r$ denotes the reduced-precision 
representation of matrix $U$ stored with precision $u_r$. 
However, this improvement is not so significant.

For the full precision GADI algorithm where all computational 
steps are performed in high precision, 
the convergence analysis has been established in \cite{doi:10.1137/21M1450197}. 
Specifically, for $\alpha_B$ in Theorem \ref{thm3}, we have:
\begin{align}
    \hat{\alpha}_B=&\gamma(\lambda+(2-\omega)\eta\hat{\kappa}_2(HS)+
    4(2-\omega)\hat{\kappa}_2(HS)u+4(2-\omega)\hat{\kappa}_2(HS)(1+u)u)\notag\\
    =&\gamma\lambda+(2-\omega)\hat{\kappa}_2(HS)(\|P_k\|+\|Q_k\|+\|P_k\|\|Q_k\|+4u+4(1+u)u),\notag
\end{align}
where we consider the case of GADI with uniform precision $u=u_f=u_r$.

For low-precision computations in GADI-IR, 
the coefficient $\alpha_B$ from Theorem \ref{thm3} can be expressed as:
\begin{align}
    \alpha_B=&\gamma(\lambda+(2-\omega)\eta\hat{\kappa}_2(H_rS_r)+
    4(2-\omega)\hat{\kappa}_2(H_rS_r)u_r+4(2-\omega)\hat{\kappa}_2(H_rS_r)(1+u_r)u_f)\notag\\
    =&\gamma\lambda+(2-\omega)\hat{\kappa}_2(H_rS_r)(\|P_k\|+\|Q_k\|+\|P_k\|\|Q_k\|+4u_r+4(1+u_r)u_f),
\end{align}
where $H_r$ and $S_r$ represent 
the reduced-precision versions of matrices $H$ and $S$ 
respectively, both stored with precision $u_r$, 
satisfying $u\leq u_f\leq u_r$. Based on equation (3.5) and Lemma
\ref{lem2}, we can establish:
\begin{align}
    \alpha_B=f(\alpha,u_r,u_f),
\end{align}
where $f(\alpha,u_r,u_f)$ is a function dependent on $\alpha$, $u_r$, and $u_f$. Based on (3.3) 
and (3.5) and considering that the improvement in (3.4) is not so significant, then 
$f(\alpha,u_r,u_f)$ exhibits monotonic behavior - decreasing with respect to 
$\alpha$ while increasing with respect to both $u_r$ and $u_f$. If we fix $\alpha$ and 
temporarily disregard precision's influence on the matrix condition number because
the improvement in (3.4) is not so significant by setting 
$\hat{\kappa}_2(HS)=\hat{\kappa}_2(H_rS_r)$, we obtain:
$$
\alpha_B=f(\alpha,u_r,u_f)\geq \hat{\alpha}_B=f(\alpha,u,u)\leq1,
$$
for $u_r\geq u, u_f\geq u$.

Consequently, $\alpha_B$ in GADI-IR could potentially exceed 1, 
leading to algorithmic divergence. Therefore, 
to ensure convergence of GADI-IR when $u\leq u_f\leq u_r$, 
it is essential to maintain $\alpha_B<1$, which according to (3.3) 
necessitates a larger value of $\alpha$.

In summary, while using lower precision for matrix $H$ can improve its condition number 
and potentially enhance the convergence rate of GADI-IR, this reduction in precision 
introduces larger values of $u_r$ and $u_f$ in equation (3.5), which may result in 
an increased $\alpha_B>1$. For GADI-IR to converge, it is crucial that $\alpha_B$ 
remains less than 1. However, the larger $\alpha_B$  
resulting from lower precision could cause $\alpha_B$ to exceed 1, leading to 
algorithmic divergence. Therefore, it becomes essential to employ the regularization 
parameter $\alpha$ to achieve an even smaller $\hat{\kappa}(H_rS_r)$, thereby 
ensuring $\alpha_B$ stays below 1 and maintaining convergence.

\subsection{Parameter prediction}
Paper\cite{doi:10.1137/21M1450197} shows that the parameter $\alpha$ is important to the 
performance of GADI. In this section, we will use mixed precision to 
accelerate the parameter prediction of GADI-IR.
\subsubsection{Parameter prediction in GADI}
The performance of GADI is sensitive to the splitting parameters. 
Paper\cite{doi:10.1137/21M1450197} proposed a data-driven parameter selection method, 
the Gaussian Process Regression (GPR) approach based on the Bayesian inference, which can 
efficiently obtain accurate splitting parameters. 
The Gaussian Process Regression (GPR) prediction process is illustrated in Figure \ref{gpr}

From Figure \ref{gpr}, it can be seen that the Gaussian Process Regression (GPR)
method established a mapping between the matrix size $n$ 
and the parameter $\alpha$. By a series known data of parameters, 
we can predict unknown parameters. The known relatively optimal parameters 
in the training data set come 
from small-scale linear systems, while the unknown parameter belongs to that of large 
linear systems. The predicted data in the training set is also used to form 
the retraining set to predict the parameter more accurately and extensively.
\subsubsection{Parameter prediction training set}
As we use Gaussian Process Regression (GPR) prediction to predict the parameter $\alpha$, we need to 
first get the training set. To get the training set, it is necessary for us 
to analysis the structure of linear system automatically to construct a 
series of small linear systems with the same structure which will take a lot of time.
To reduce the time consumption, we put this progress in FP32 low precision. Then we 
will get a series of small linear systems with the same structure of the original 
linear system. Afterwards we use the dichotomy to find a 
series of $\{\alpha_k\}$ with these low scale linear systems in FP32 precision.
\begin{figure}[!t]
    \centerline{\includegraphics[width=\columnwidth]{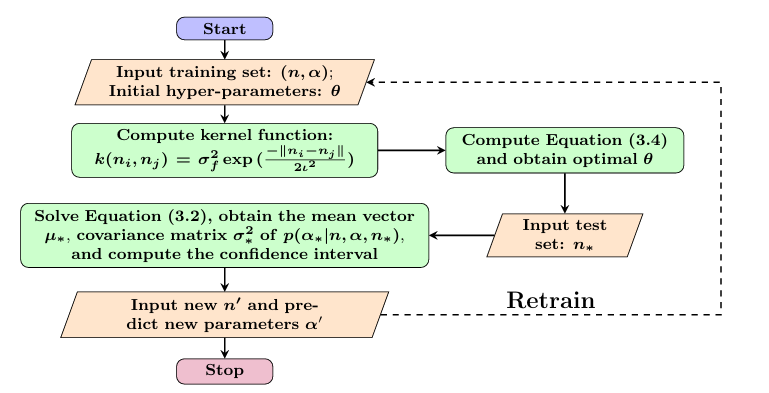}}
    \caption{\textit{Flow chart of Gaussian Process Regression (GPR) parameters prediction.}}
    \label{gpr}
\end{figure}

Finally, we will use the calculated training set $\{\alpha_k\}$ to do 
Gaussian Process Regression (GPR) prediction with FP32 precision. 

By using low precision to find the training set, we can get 
the best parameter $\alpha$ with less time of the original 
implementation without losing performance in GADI-IR with parameter prediction 
which is illustrated in Table \ref{pred32}. The result of Table \ref{pred32} 
is tested on experiment three-dimensional convection-diffusion equation in 
section~\ref{subsec:3de}.

\begin{table}[htbp]
    \centering
    \begin{tabular}{|c|c|c|}
        \hline
        \textbf{$n^3$} & \textbf{$\alpha$ FP64} & \textbf{$\alpha$ FP32} \\
        \hline
        $32^3$  & 0.0699 & 0.0699 \\
        $64^3$ & 0.0599 & 0.0599 \\
        $128^3$ & 0.0595 & 0.0595 \\
        \hline
    \end{tabular}
    \caption{\textit{$\alpha$ predicted by different precision training set for 
    three-dimensional convection-diffusion equation.}}
    \label{pred32}
\end{table}


\section{Numerical Experiments}
\label{sec:ne}
In this section, we will evaluate our algorithm using 
various mixed precision configurations to validate our analysis.

\subsection{Three-dimensional convection-diffusion equation}
\label{subsec:3de}
Consider 3D convection-diffusion equation:
\begin{equation}
    -(u_{x_1x_1}+u_{x_2x_2}+u_{x_3x_3})+
    (u_{x_1}+u_{x_2}+u_{x_3})=f(x_1,x_2,x_3)
\end{equation}
on the unit cube $\Omega =[0,1]\times[0,1]\times[0,1]$ with Dirichlet boundary condition. 
By using the centered difference method to discretize the convective-diffusion equation, we 
can obtain the linear sparse system $Ax=b$. The coefficient matrix $A$ is :
$$
A=T_{x} \otimes I \otimes I+I \otimes T_{y} \otimes I+I \otimes I \otimes T_{z}, 
$$
where $T_{x}, T_{y}, T_{z}$ are the tridiagonal matrices. $T_{x}=\mathrm{Tridiag}(t_2,t_1,t_3),
T_{y}=T_{z}=\mathrm{Tridiag}(t_2,0,t_3),t_1=6,t_2=-1-r,t_3=-1+r,r=1/(2n+2)$. $n$ is 
the degree of freedom in each direction. $x \, \in\, \mathbb{R}^{n^{3}} $ is 
the solution vector and $b \, \in\, \mathbb{R}^{n^{3}} $ is the right-hand side vector which 
is generated by choosing the exact solution $x=(1,1,...,1)^T$. The relative error  
is defined as $\mathrm{RES}=\| r^{( k )} \|_{2} / \| r^{( 0 )} \|_{2} $.
All tests 
are started from the zero vector. $r^{( k )}=b-A x^{( k )} $ is the $k$-th step residual.

The GADI-IR algorithm is tested on the 3D convection-diffusion equation with 
\begin{equation}
    H=\frac{A+A^*}{2}, S=\frac{A-A^*}{2}.
\end{equation}
\cite{articleSuccessive-Overrelaxation}
splitting strategy. The results of the numerical experiments are shown in Table \ref{3d}. 
The table presents the relative residuals (RRES) for the 3D 
convection-diffusion equation using different combinations 
of precisions for the components $u_r$, $u$, and $u_f$. 
The experiments demonstrate that using double precision 
for all components consistently achieves the lowest residuals. 
In contrast, using half precision for $u_r$ results in 
higher residuals. 
This outcome aligns with our error analysis of GADI-IR in 
section~\ref{sec:ea}, 
which predicts the convergence of GADI-IR and 
increased sensitivity and potential 
instability when lower precision is employed for critical computations 
particularly when $\alpha$ is set to lower values.
The results highlight the sensitivity of the algorithm's performance 
to the choice of precision and the regularization parameter $\alpha$.

\begin{figure}[H]
    \centering
    \begin{minipage}[t]{0.45\textwidth} 
        \centering
        \begin{tabular}{ccccc}
        \toprule
        \cmidrule(lr){4-4}
        $u_r$ & $u$ & $u_f$ & $\alpha$ & RRES \\
        \midrule
        double  & double & double & $0.01$ &  $10^{-13}$\\
        single  & single & single & $0.01$ & $10^{-4}$\\
        single  & double & double & $0.01$ & $10^{-13}$\\
        half  & double & single & $0.01$ &  $-$\\
        half  & double & double & $0.01$ &  $-$\\
        half  & single & single   & $0.01$ & $10^{-4}$\\
        \midrule
        double  & double & double & $0.02$ &  $10^{-13}$\\
        single  & single & single & $0.02$ & $10^{-4}$\\
        single  & double & double & $0.02$ & $10^{-13}$\\
        half  & double & single & $0.02$ &  $10^{-8}$\\
        half  & double & double & $0.02$ &  $10^{-8}$\\
        half  & single & single   & $0.02$ & $10^{-4}$\\
        \midrule
        double  & double & double & $10.0$ &  $10^{-13}$\\
        single  & single & single & $10.0$ & $10^{-4}$\\
        single  & double & double & $10.0$ & $10^{-13}$\\
        half  & double & single & $10.0$ &  $10^{-10}$\\
        half  & double & double & $10.0$ &  $10^{-10}$\\
        half  & single & single   & $10.0$ & $10^{-4}$\\
        \bottomrule
        \end{tabular}
        \captionof{table}{\textit{Relative Residual with different precisions 
        for 3D convection-diffusion equation.}}
        \label{3d}
    \end{minipage}%
    \hfill 
    \begin{minipage}[t]{0.45\textwidth} 
        \centering
        \begin{tabular}{ccccc}
        \toprule
        $u_r$ & $u$ & $u_f$ & $\alpha$ & RRES \\
        \midrule
        double  & double & double & $0.01$ &  $10^{-9}$\\
        single  & single & single & $0.01$ & $10^{-2}$\\
        single  & double & double & $0.01$ & $10^{-9}$\\
        half  & double & single & $0.01$ &  $-$\\
        half  & double & double & $0.01$ &  $-$\\
        half  & single & single   & $0.01$ & $10^{-3}$\\
        \midrule
        double  & double & double & $0.02$ &  $10^{-9}$\\
        single  & single & single & $0.02$ & $10^{-2}$\\
        single  & double & double & $0.02$ & $10^{-9}$\\
        half  & double & single & $0.02$ &  $10^{-5}$\\
        half  & double & double & $0.02$ &  $10^{-5}$\\
        half  & single & single   & $0.02$ & $10^{-3}$\\
        \midrule
        double  & double & double & $10.0$ &  $10^{-9}$\\
        single  & single & single & $10.0$ & $10^{-2}$\\
        single  & double & double & $10.0$ & $10^{-9}$\\
        half  & double & single & $10.0$ &  $10^{-6}$\\
        half  & double & double & $10.0$ &  $10^{-6}$\\
        half  & single & single   & $10.0$ & $10^{-3}$\\
        \bottomrule
        \end{tabular}
        \captionof{table}{\textit{Relative Residual with different precisions 
        for CARE.}}
        \label{care}
    \end{minipage}
\end{figure}



From Table \ref{3d}, It can be seen that the mixed precision algorithm GADI-IR 
with $u_r=\mathrm{half},u=u_f=\mathrm{double}$ exhibits significant 
sensitivity to the regularization parameter 
$\alpha$. This mixed precision strategy, where the 
reduced precision (FP16) is used for the iterative 
process and the higher precision (double) 
is retained for key updates and parameter calculations, 
provides considerable performance benefits in 
terms of computation speed and memory usage. However, the choice of 
$\alpha$, which controls the balance between 
the regularization and the solution accuracy, 
plays a crucial role in ensuring the stability 
and convergence of the algorithm.

Based on this mixed precision strategy 
$(u_r=\mathrm{half},u=u_f=\mathrm{double})$, we 
conducted experiments to investigate the impact of varying 
$\alpha$ on the convergence residual 
RES. The results, as shown in the Figure \ref{fig:alpha-impact} and 
\ref{fig:care-alpha-impact}, 
illustrate a clear trend: as the regularization parameter 
$\alpha$ increases, the convergence performance 
improves for this specific test case.


This behavior indicates that larger values of 
$\alpha$ enhance the stability of the iterative process, 
reducing the impact of precision-related errors 
and promoting more reliable convergence to the desired solution. 
The experimental results suggest that 
$\alpha$ effectively mitigates the potential 
inaccuracies introduced by the mixed precision approach just as
the theoretical analysis of regularization in section~\ref{subsec:reg} and
~\ref{subsec:breg} predicted. 

While increasing the regularization parameter 
$\alpha$ generally improves the convergence residual 
RES, it may also lead to a higher number of 
iterations required for convergence. This is because larger values of 
$\alpha$ can over-regularize the system, 
effectively damping the iterative process and 
slowing down the overall rate of convergence. 
Table \ref{tab:alpha-effects} shows how different values of the 
regularization parameter 
$\alpha$ influence the residual 
RES and the number of iteration steps. Smaller or larger 
$\alpha$ values lead to a significant increase 
in iteration steps, while moderate 
$\alpha$ values result in fewer steps and smaller residuals. 
This indicates that the choice of 
$\alpha$ is crucial for balancing efficiency and accuracy.
\begin{table}[htbp]
    \centering
    \sisetup{scientific-notation=true} 
    \begin{tabular}{c c c}
        \toprule
        \textbf{Alpha ($\alpha$)} & \textbf{Residual (res)} & \textbf{Iteration Steps} \\
        \midrule
        0.01 & $-$ & $-$ \\
        0.02 & 4.86e-08 & 2126 \\
        0.05 & 3.10e-08 & 465  \\
        0.1  & 1.88e-08 & 150  \\
        0.5  & 4.85e-09 & 75   \\
        1    & 2.47e-09 & 120  \\
        5    & 4.96e-10 & 468  \\
        10   & 3.51e-10 & 909  \\
        100  & 3.51e-11 & 8862 \\
        \bottomrule
    \end{tabular}
    \caption{\textit{Impact of Regularization Parameter 
    $\alpha$ on Convergence Residuals and Iteration Steps.}}
    \label{tab:alpha-effects}
\end{table}

As a result, selecting the optimal 
$\alpha$ involves balancing the 
trade-off between minimizing the residual 
error and controlling the computational cost 
associated with additional iterations. 
For practical applications, it is essential to choose 
$\alpha$ based on the specific problem 
characteristics and acceptable computational overhead.

By carefully tuning 
$\alpha$ within a reasonable range, 
one can achieve a compromise that maintains 
a low residual error while avoiding excessive 
iteration counts, thereby optimizing both accuracy 
and efficiency under the constraints of the mixed 
precision strategy. Adaptive or problem-specific 
strategies for determining 
$\alpha$ could further enhance the robustness and 
practicality of the algorithm in diverse scenarios.

\begin{figure}
    \centering
    \begin{minipage}[t]{0.45\textwidth}
        \centering
        \includegraphics[width=0.8\textwidth]{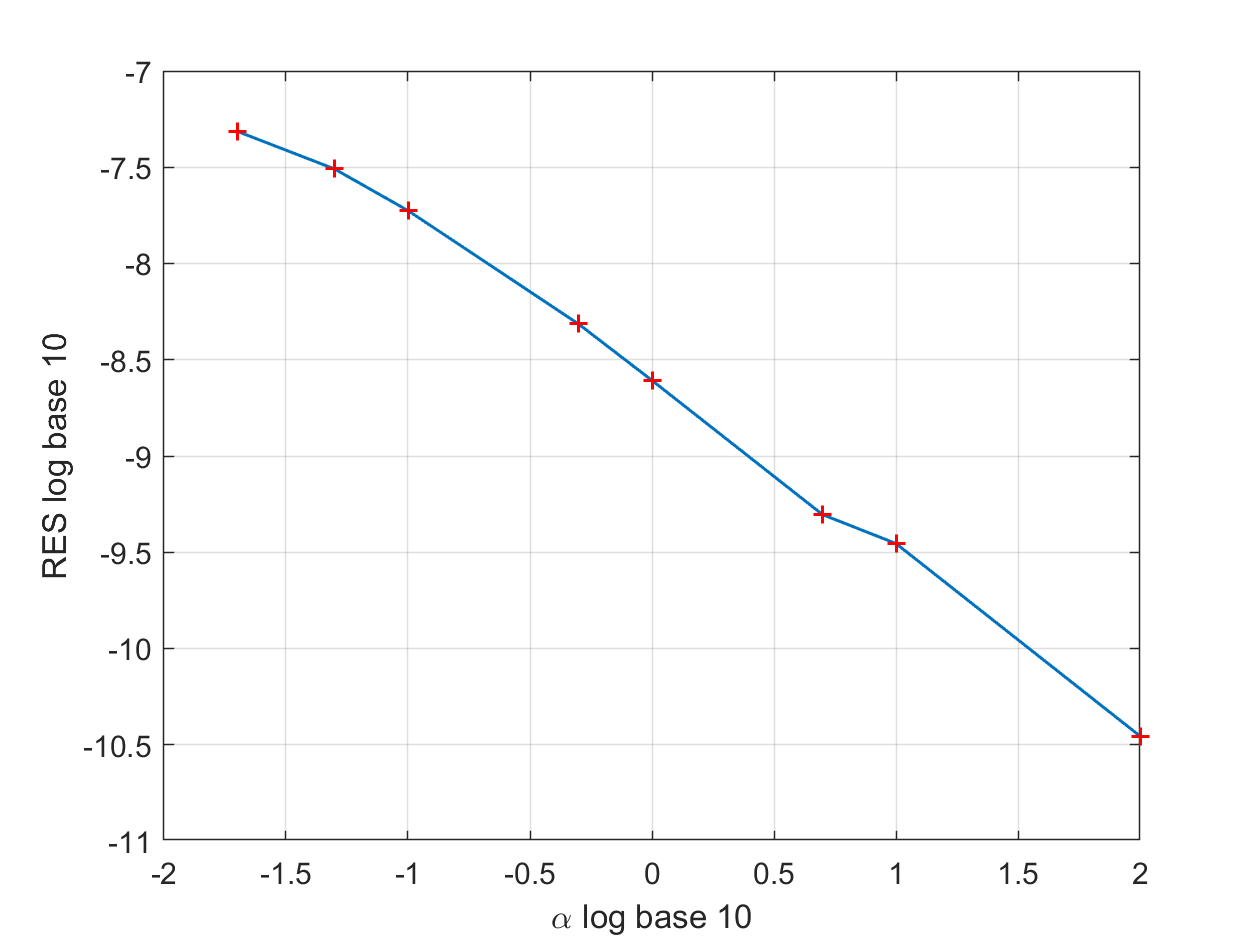} 
        \caption{\textit{Impact of $\alpha$ on convergence residual of 3d convection-diffusion equation.}} 
        \label{fig:alpha-impact} 
    \end{minipage}
    \hfill 
    \begin{minipage}[t]{0.45\textwidth}
        \centering
        \includegraphics[width=0.8\textwidth]{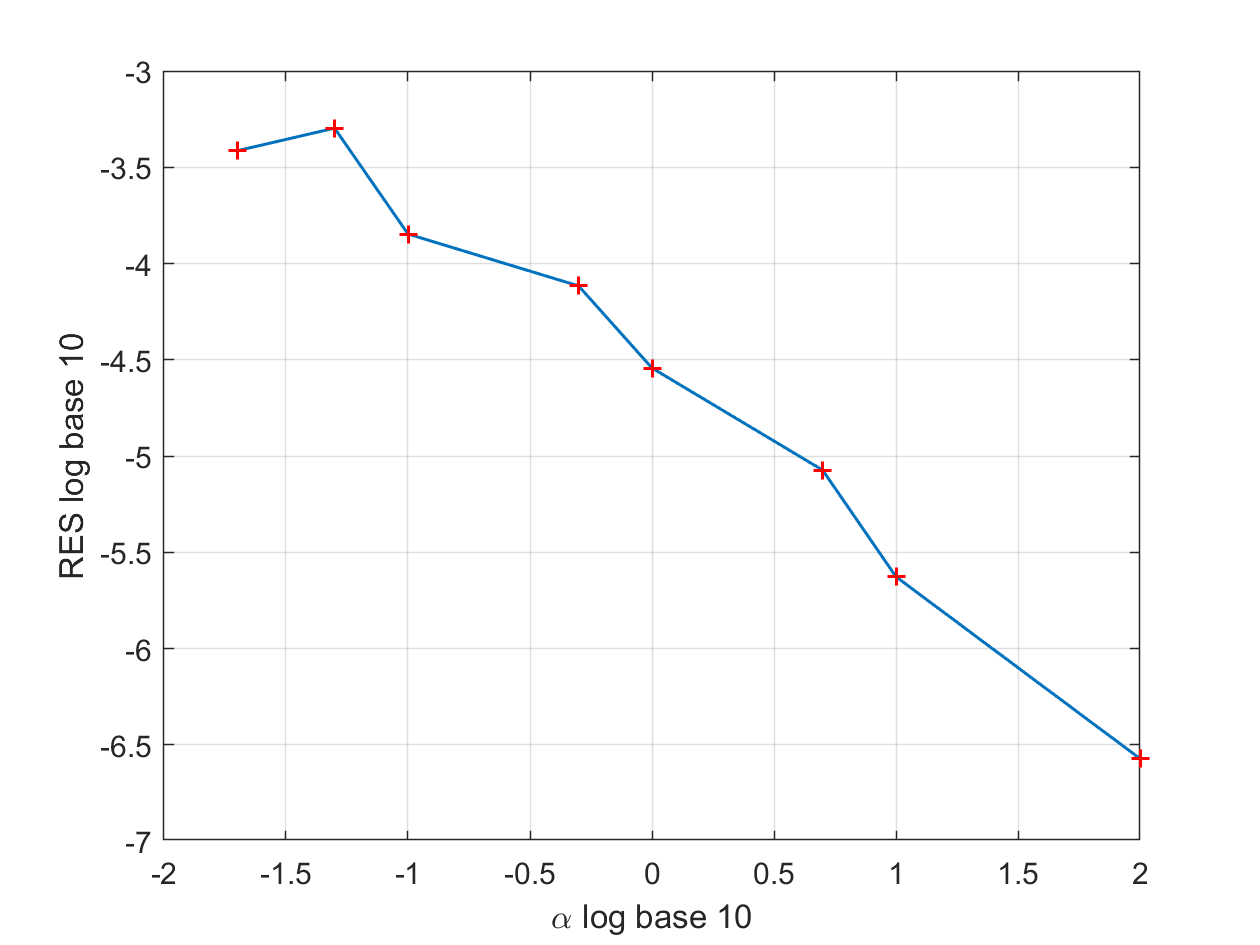} 
        \caption{\textit{Impact of $\alpha$ on convergence residual of CARE.} 
        \label{fig:care-alpha-impact}} 
    \end{minipage}
\end{figure}

\subsection{CARE equation}

Next, we apply our algorithm to other problem. We test our algorithm on 
CARE equation which is a typical problem in control theory. In this section, 
we consider the continuous-time algebraic Riccati equation(CARE): 
\begin{align}
    A^TX+XA-XKX+Q=0.
\end{align}
where $A, \, K, \, Q \in\mathbb{C}^{n \times n}, \, K=K^{*}, \, Q=Q^{*}, $ and 
$X$ is an unknown matrix. Paper \cite{JIANG2025642} proposed an algorithm named 
Newton-GDAI which is based on GADI to solve this equation. In this algorithm, 
GADI is used to solve Lyapunov equation in inner iteration where we apply our 
mixed precision algorithm GADI-IR. 
In this equation, complex matrices are used where 
we follow the work of \cite{8948697} to handle 
complex arithmetic in half precision. 
Their research provides effective strategies 
for mixed-precision computation with complex matrices on GPUs, 
which is essential for our implementation.
The results of the numerical experiments, 
as shown in Table \ref{care}, 
indicate that the mixed precision 
algorithm GADI-IR demonstrates varying 
levels of residual error (RRES) depending 
on the precision levels used for different 
components ($u_r$, $u$, $u_f$) and 
the regularization parameter $\alpha$. 
Specifically, using double precision for 
all components consistently achieves the 
lowest residual errors across different values 
of $\alpha$. In contrast, using half precision 
for $u_r$ results in higher residual errors, 
particularly when $\alpha$ is set to lower values. 
The experiments highlight the sensitivity of the algorithm's 
performance to the choice of precision and the 
regularization parameter, emphasizing the need for 
careful selection to balance computational efficiency and accuracy.


Similarly, the impact of the regularization parameter $\alpha$ on the convergence residual 
RES for the mixed precision strategy with $u=u_f=double$ and $u_r=single$ in the CARE problem 
is illustrated in Figure \ref{fig:care-alpha-impact}, akin to the analysis for the 
3D convection-diffusion equation discussed in section~\ref{subsec:3de}.

\subsection{Sylvester equation}

To further test the performance of our algorithm, we apply mixed GADI-IR method to continuous 
Sylvester equation\cite{BENNER20091035}. The continuous Sylvester equation can be written as:
\begin{align}
    AX+XB+C=0.
\end{align}
where $A \, \in\, \mathbb{C}^{m \times m}, B \, \in\, \mathbb{C}^{n \times n} \, \mathrm{~ a n d ~} \, C \, \in\, \mathbb{C}^{m \times n} $, are 
sparse matrices. $X \, \in\, \mathbb{C}^{m \times n} $ is the unknown matrix. 
Applying the  mixed precision algorithm GADI-IR to continuous Sylvester equation and 
replacing splitting matrices $M,N$ with $A,B$ respectively, we can obtain the 
mixed GADI-AB method.

The sparse matrices $A, B$ have the following structure:
$$
A=B=M+2rN+\frac{100}{(n+1)^2}I.
$$
where $r$ is a parameter which controls Hermitian dominated or skew-Hermitian dominated of 
matrix. $M, \, N \in\mathbb{C}^{n \times n} $ are tridiagonal matrices 
$M=\mathrm{T r i d i a g} (-1, 2,-1 ), N \,=\, \mathrm{T r i d i a g} ( 0. 5, 0,-0. 5 ) $. 
We apply mixed GADI-AB to solve the Sylvester equations for $r=0.01,0.1,1$, RES is calculated 
as $R^{( k )}=C-A X^{( k )}-X^{( k )} B $.

The numerical experiments are shown in Table~\ref{tab:sylvester-results} which presents 
numerical test results 
for the Sylvester equation under different parameter combinations of 
$r$ and 
$\alpha$, with solution accuracy evaluated by the residual (res). 
The results show that the residuals range from 
1e-8 
to 
1e-10
, indicating high numerical accuracy of the solutions. 
The method demonstrates stability and reliability 
across various parameter settings.

Additionally, the impact of the regularization parameter 
$\alpha$ on the convergence residual of the Sylvester equation 
with different $r$ values 
is illustrated in Figure~\ref{fig:sylvester-alpha-impact}. 
It can be observed that the convergence residual decreases 
as $\alpha$ increases, with the rate of convergence 
improving for larger $\alpha$ values. This behavior 
is consistent with the results of the previous tests,
indicating that the regularization parameter 
$\alpha$ plays a crucial role in balancing the 
precision-related errors and promoting the 
convergence of the mixed precision GADI-IR algorithm.

\begin{table}[htbp]
    \centering
    \sisetup{scientific-notation=true}
    \begin{tabular}{c c c}
        \toprule
        \textbf{$r$} & \textbf{$\alpha$} & \textbf{Residual (res)} \\
        \midrule
        0.01 & 0.01 & $-$ \\
            & 0.02 & 1.563e-08 \\
             & 10   & 2.4092e-10 \\
        0.1  & 0.01 & $-$ \\
            & 0.02 & 1.563e-08 \\
             & 10   & 4.2351e-10 \\
        1    & 0.01 & $-$ \\
            & 0.02 & 1.563e-08 \\
             & 10   & 3.2551e-10 \\
        \bottomrule
    \end{tabular}
    \caption{\textit{Test results of the Sylvester equation for different $r$ and $\alpha$ values.}}
    \label{tab:sylvester-results}
\end{table}

\begin{figure}[htbp]
    \centering
    \begin{minipage}[t]{0.3\textwidth}
        \centering
        \includegraphics[width=\textwidth]{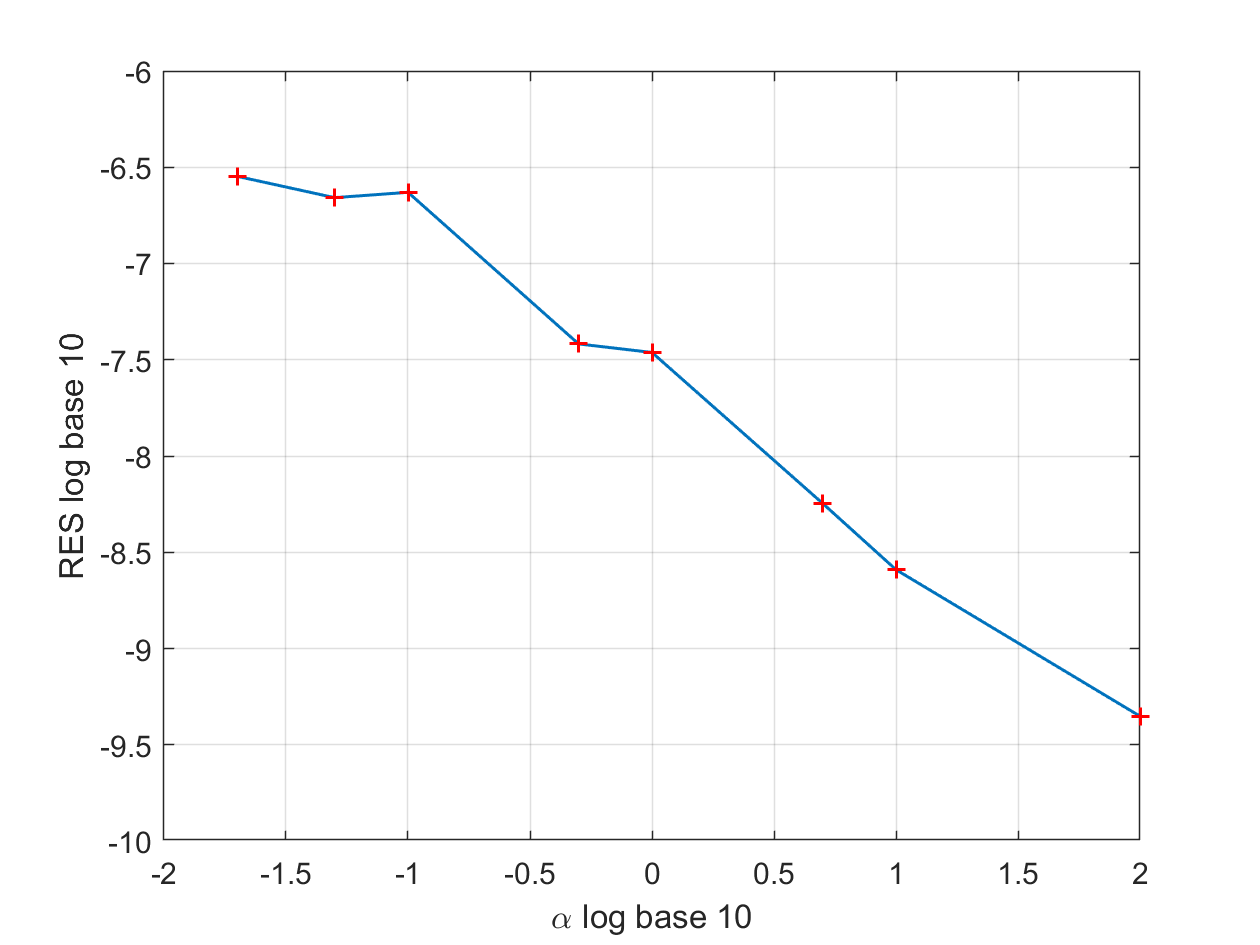} 
        \caption*{(a) $r = 0.01$} 
    \end{minipage}
    \hfill 
    \begin{minipage}[t]{0.3\textwidth}
        \centering
        \includegraphics[width=\textwidth]{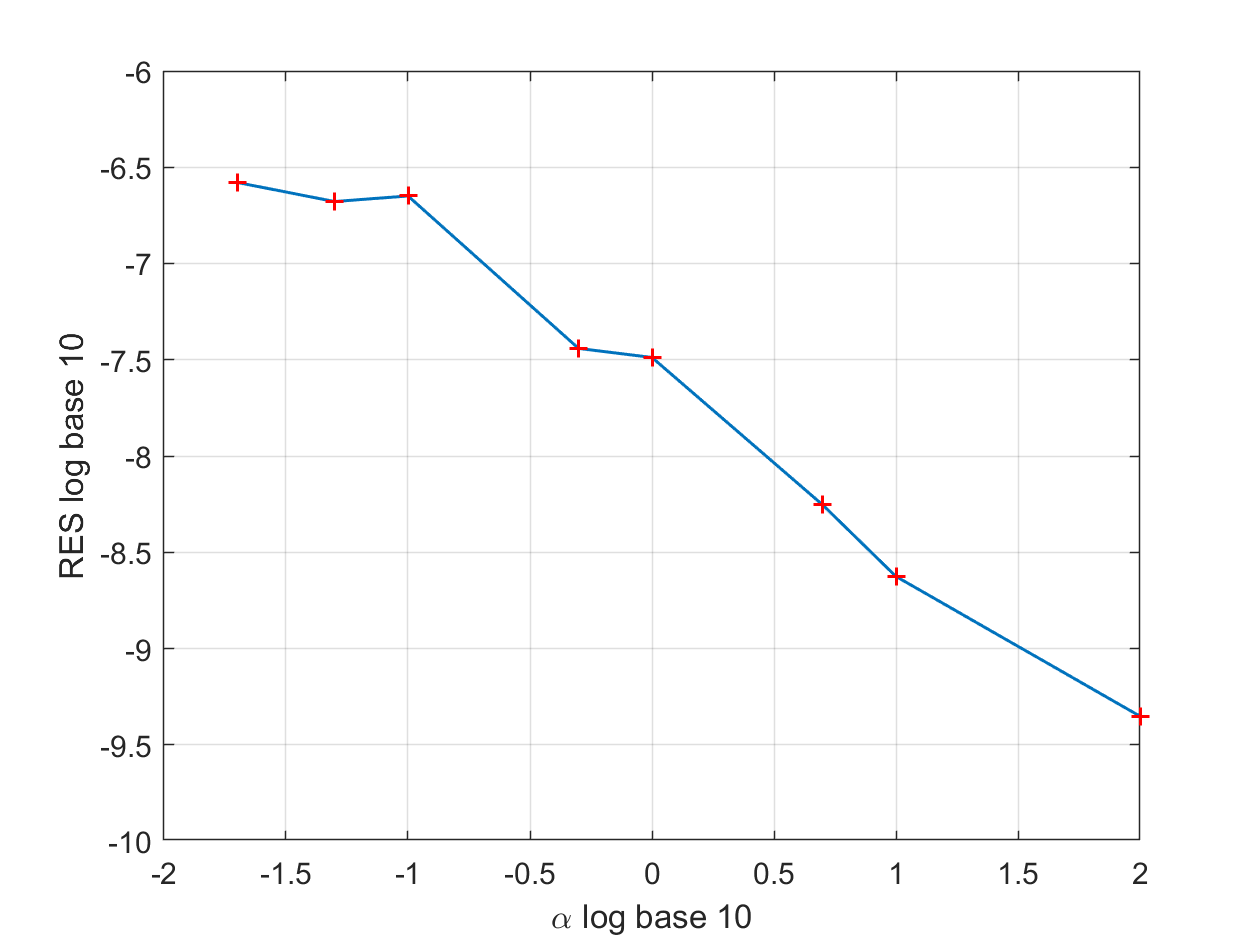} 
        \caption*{(b) $r = 0.1$} 
    \end{minipage}
    \hfill 
    \begin{minipage}[t]{0.3\textwidth}
        \centering
        \includegraphics[width=\textwidth]{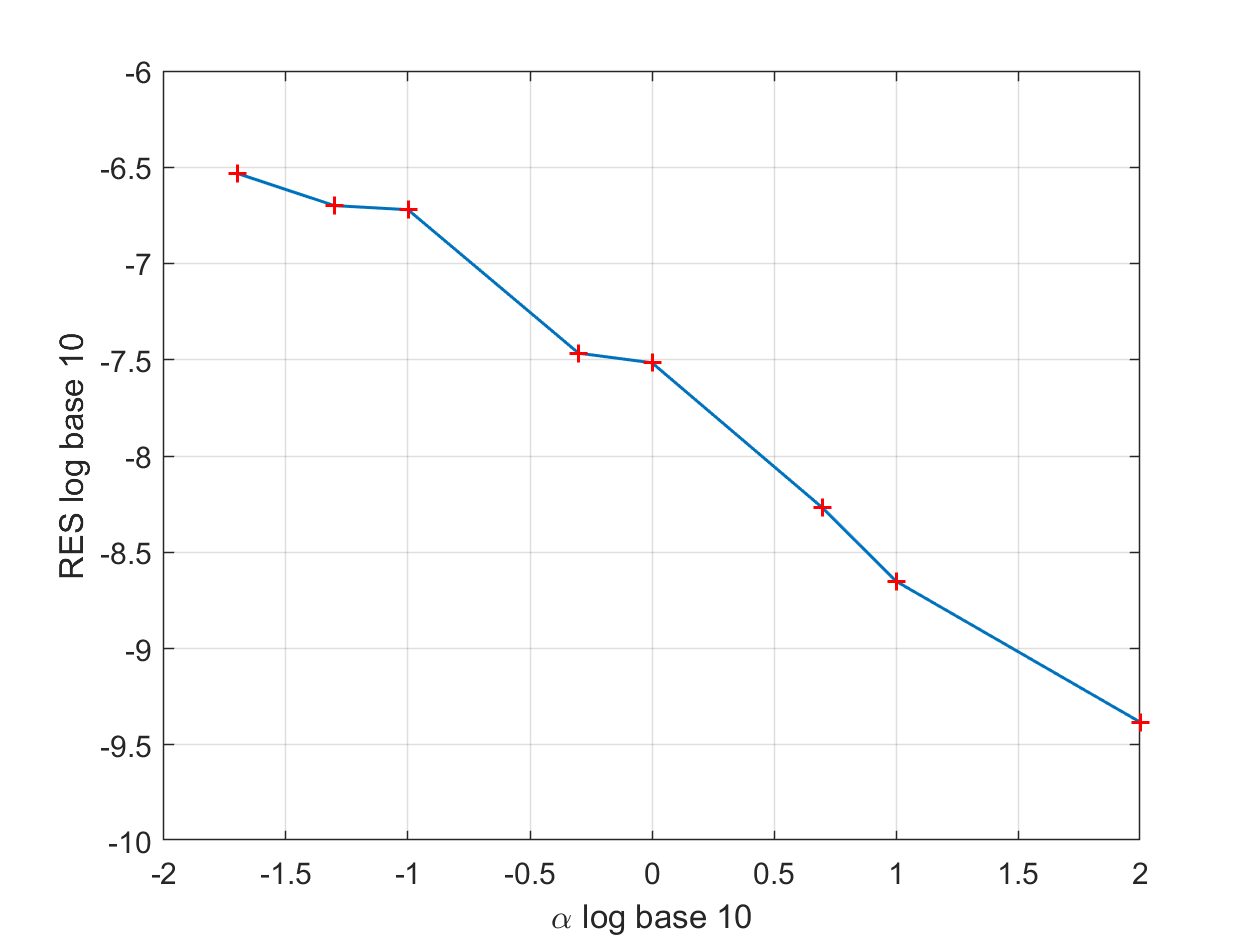} 
        \caption*{(c) $r = 1$} 
    \end{minipage}
    \caption{\textit{Impact of $\alpha$ on convergence residual of Sylvester equation with different $r$.} 
    \label{fig:sylvester-alpha-impact}} 
\end{figure}

\section{Conclusion and Future work}
\label{sec:con}
In this paper, we have presented a novel 
mixed-precision iterative refinement algorithm, 
GADI-IR, which effectively 
combines multiple precision arithmetic to 
solve large-scale sparse linear systems. 
Our key contributions and findings can be summarized as follows:

First, we developed a theoretical framework 
for analyzing the convergence of mixed-precision GADI-IR, 
establishing the relationship between precision levels and 
convergence conditions. Through careful backward error analysis, 
we demonstrated how the regularization parameter $\alpha$ can 
be used to ensure convergence when using reduced precision arithmetic.

Second, we successfully integrated low-precision computations 
into the parameter prediction process using Gaussian Process Regression (GPR), achieving 
significant speedup without compromising accuracy. Our experimental 
results showed that using FP32 for the training set generation 
reduces the computational time by approximately 50\% while maintaining 
prediction quality comparable to FP64.

Third, comprehensive numerical experiments on various test 
problems, including three-dimensional convection-diffusion 
equations and Sylvester equations, validated both our theoretical 
analysis and the practical effectiveness of the mixed-precision 
approach. The results demonstrated that GADI-IR can achieve 
substantial acceleration while maintaining solution accuracy 
through appropriate precision mixing strategies.

Looking ahead, we have several promising directions for future research like 
investigation of newer precision formats (e.g., FP8) that 
could potentially further enhance the algorithm's performance and efficiency and
implementation of GADI-IR on modern supercomputing platforms.

These future developments will further strengthen the practical 
applicability and efficiency of the GADI-IR algorithm in 
solving large-scale linear systems.



\bibliographystyle{siamplain}
\bibliography{references}

\end{document}